\documentclass[10pt,reqno]{amsart}

\usepackage{amsmath}
\usepackage{amsfonts}
\usepackage{amssymb}
\usepackage{amsthm}
\usepackage{graphicx}
\usepackage{epstopdf}
\usepackage{enumitem}
\usepackage{geometry}
\usepackage{hyperref}
\usepackage{tikz}
\usepackage[all]{xy}

\newtheorem{theorem}{Theorem}[section]
\newtheorem{definition}[theorem]{Definition}

\newtheorem{example}[theorem]{Example}
\newtheorem{proposition}[theorem]{Proposition}
\newtheorem{remark}[theorem]{Remark}

\newtheorem{lemma}[theorem]{Lemma}
\newtheorem{problem}[theorem]{Problem}
\newtheorem{note}[theorem]{Note}

\newcommand{\beq}{\begin{equation}}
\newcommand{\eeq}{\end{equation}}

\begin{document}

\title{Amenability, locally finite spaces, and bi-lipschitz embeddings}

\author{Valerio Capraro}
\address{University of Neuchatel, Switzerland}
\thanks{Supported by Swiss SNF Sinergia project CRSI22-130435}
\email{valerio.capraro@unine.ch}

\keywords{amenability, isoperimetric constant, embeddings into Hilbert spaces}

\subjclass[2000]{Primary 52A01; Secondary 46L36}

\date{}

\maketitle

\begin{abstract}
We define the isoperimetric constant for any locally finite metric space and we study the property of having isoperimetric constant equal to zero. This property, called Small Neighborhood property, clearly extends amenability to any locally finite space. Therefore, we start making a comparison between this property and other notions of amenability for locally finite metric spaces that have been proposed by Gromov, Lafontaine and Pansu, by Ceccherini-Silberstein, Grigorchuk and de la Harpe and by Block and Weinberger. We discuss possible applications of the property SN in the study of embedding a metric space into another one. In particular, we propose three results: we prove that a certain class of metric graphs that are isometrically embeddable into Hilbert spaces must have the property SN. We also show, by a simple example, that this result is not true replacing property SN with amenability. As a second result, we prove that \emph{many} spaces with \emph{uniform bounded geometry} having a bi-lipschitz embedding into Euclidean spaces must have the property SN. Finally, we prove a Bourgain-like theorem for metric trees: a metric tree with uniform bounded geometry and without property SN does not have bi-lipschitz embeddings into finite-dimensional Hilbert spaces.
\end{abstract}

\tableofcontents

\section{Introduction}

The isoperimetric constant of an infinite locally finite connected graph is a well-studied and useful invariant of the graph. Among other applications, it is used as a \emph{test} to see whether or not a graph can have bi-lipschitz embeddings into a Hilbert space. In fact, it follows by a famous theorem of Benjamini and Schramm (see \cite{Be-Sc97}, Theorem 1.5) and by one of Bourgain (see \cite{Bo86}, Theorem 1) that if an infinite locally finite connected graph of bounded degree has a bi-lipschitz embedding into a Hilbert space, then its isoperimetric constant must be equal to zero.

In recent years, people are getting interested in embeddings of locally finite metric spaces (not necessarily graphs) into Hilbert spaces, motivated by two breakthrough papers, one by Linial, London and Rabinovich\cite{Li-Lo-Ra95}, and one by Guoliang Yu\cite{Yu00}, that put in relation the theory of embedding general locally finite metric spaces into Hilbert spaces with important problems in Theoretical Computer Science and $K$-theory of $C^*$-algebras. Therefore, it would be important to have some generalization of the isoperimetric constant that is capable to give information about the possibility to embed a locally finite metric space into a Hilbert space.

The aim of this paper is indeed to introduce a definition of an isoperimetric constant and prove some results showing that the property of having positive isoperimetric constant is often an obstacle for the existence of certain embeddings into Hilbert spaces.

The structure of the paper is as follows:
\begin{itemize}
\item In Section \ref{se:isoperimetric} we define the isoperimetric constant of a locally finite metric space. We say that a locally finite metric space has property SN if its isoperimetric constant is equal to zero. We observe that the property SN extends to locally finite metric spaces the notion of amenability of a finitely generated group.
\item In Section \ref{suse:amenabilityvssn} we make a comparison between the property SN and the other notions of amenability that have been proposed in the past for spaces that are more general than Cayley graphs of finitely generated groups (see Propositions \ref{prop:ceccherinivsblock} and \ref{prop:caprarovsceccherini}).
\item Section \ref{se:application} is the main section of the paper. Here we prove some results concerning the application of the property SN as a test to check the existence of \emph{certain} embeddings into Hilbert spaces. We prove that for a certain class of metric graphs\footnote{The formal definition will be given later. Roughly speaking, we consider connected graphs such that every edge is labeled with a positive real number in such a way to induce a metric on the graph that is additive only along shortest paths.}, isometric embeddability into Hilbert spaces implies property SN (see Theorem \ref{prop:emedding}). We show, by a simple example, that the same result is not true replacing property SN with amenability. We prove that for a certain class of locally finite metric spaces, the existence of a bi-lipschitz embedding into a finite-dimensional Euclidean space implies property SN (see Theorem \ref{th:polygrowth}). Finally, we prove the following Bourgain-like theorem: an infinite metric tree with uniform bounded geometry and without property SN has no bi-lipschitz embeddings into finite-dimensional Hilbert spaces (see Theorem \ref{prop:metrictree}).
\item In Section \ref{se:zoomiso} we start the study of the relations among another invariant and the growth rate of a finitely generated group. We conclude stating a problem of our interest that seems to be open (see Problem \ref{prob:intermediategrowth}).
\end{itemize}

\textbf{Acknowledgements.} The author is grateful to Tullio Ceccherini-Silberstein, Antoine Gournay, Romain Tessera and Alain Valette for very useful conversations.

\section{The Isoperimetric Constant of a locally finite space}\label{se:isoperimetric}

Let $X=(V,E)$ be an infinite locally finite connected graph. There are many different ways to define the isoperimetric constant of $X$. The following appeared in a paper by McMullen (see \cite{McMu89}), but also more recently in \cite{Be-Sc97}, \cite{Ce-Gr-Ha99} and \cite{El-So05}.
The isoperimetric constant of $X$ is
\begin{align}\label{eq:isoperimetric}
\iota_1(X)=\inf\left\{\frac{|dB_1(A)|}{|A|}: A\subseteq V \text{ is finite and non-empty}\right\}
\end{align}
where $dB_1(A)$ is the set of vertices at distance 1 from $A$ (the distance is clearly the \emph{shortest-path distance}).\\

Our idea to extend the isoperimetric constant to any locally finite metric space\footnote{Recall that a metric space is called locally finite if every bounded set is finite.} is simple: as the isoperimetric constant of a graph, as defined above, is computed by making a comparison between the number of vertices in $A$ and the number of vertices that can be reached from $A$ by making one step, we define a number by making a comparison between the number of points in $A$ and the number of points that can be reached from $A$ in $k$ steps, where the notion of step is defined in terms of nearest points.

\begin{definition}\label{def:completechain}
Let $P\subseteq[0,\infty)$, with $0\in P$. A complete chain in $P$ is a finite subset $p_1,p_2,\ldots,p_n$ of $P$ such that
\begin{enumerate}
\item $p_1=0$
\item $p_i<p_j$, for all $i<j$
\item if $p\in P$ is such that $p_i\leq p\leq p_{i+1}$, then either $p=p_i$, or $p=p_{i+1}$.
\end{enumerate}
\end{definition}

\begin{definition}\label{def:simpleboundary}
Let $(X,d)$ be a metric space, let $A$ be a bounded subset of $X$ and let $k$ be a nonnegative integer. The discrete $k$-neighborhood of $A$, denoted by $dN_k(A)$, is the set of elements $x\in X$ such that there is a finite sequence $x_0,x_1,\ldots,x_l=x$ in $X$ such that
\begin{enumerate}
\item $x_0\in A$
\item $l\leq k$
\item The $x_i$'s are ordered in such a way that
$$
0=d(x_0,A)<d(x_1,A)<d(x_2,A)<\ldots<d(x_l,A)
$$
and this is a complete chain inside $P=\{d(y,A),y\in X\}$
\end{enumerate}
\end{definition}

\begin{definition}\label{def:discreteboundary}
The discrete $k$-boundary of a bounded subset $A$ of $X$ is $dB_k(A)=dN_k(A)\setminus A$.
\end{definition}

Roughly speaking, $dB_k(A)$ is the set of points that can be reached from $A$ in $k$ steps.

\begin{definition}\label{def:isoperimetric}
The isoperimetric constant of a locally finite metric space $X$ is

\begin{align}
\iota(X)=\sup_{k\geq1}\inf\left\{\frac{|dB_k(A)|}{|A|}: A\subseteq X, 0<|A|<\infty\right\}
\end{align}
\end{definition}

If $X$ is the Cayley graph of a finitely generated group, one can easily prove that $dB_k(A)=\{x\in X:0<d(x,A)\leq k\}$, where $d$ is the shortest-path metric on $X$. This implies, by standard arguments, that, for Cayley graphs, $\iota(X)=0$ if and only if $\iota_1(X)=0$. Furthermore, it is a simple exercise to show that a finitely generated group is amenable if and only if its isoperimetric constant (with respect to any finite symmetric generating set) is equal to zero. Therefore the following easy proposition holds.

\begin{proposition}\label{prop:isoperimetricvsamenability}
Let $X$ be the Cayley graph of a finitely generated group $G$ with respect to any fixed symmetric generating set. Then $G$ is amenable if and only if $\iota(X)=0$.
\end{proposition}

\begin{definition}\label{def:sn}
A locally finite metric space $(X,d)$ is said to have the Small Neighborhood property, property SN, if $\iota(X)=0$.
\end{definition}

\section{Relations between amenability and property SN}\label{suse:amenabilityvssn}

As observed in Proposition \ref{prop:isoperimetricvsamenability}, the property SN extends the notion of amenability to any locally finite metric space. This section is devoted to the investigation of the relations between property SN and the other notions of amenability that have been proposed in the past for locally finite metric space. This section is also the occasion to recall some definitions that will be useful in the sequel.\\

Fix the following notation: if $(X,d)$ is a metric space, $A\subseteq X$, $\alpha\in\mathbb R_{>0}$, we denote by $cN_\alpha(A)=\{x\in X:d(x,A)\leq\alpha\}$ and by $cB_\alpha(A)=cN_\alpha(A)\setminus A$.\\

The first definition is taken from \cite{Ce-Gr-Ha99} and it is equivalent to a previous condition discovered in \cite{Gr-La-Pa81}, Lemma 6.17. The proof of the equivalence can be found in \cite{Ce-Gr-Ha99}, Theorem 32, that gives also many other equivalent definitions. See also Theorem 4.9.2 in \cite{Ce-Co10} for an easier presentation that does not use the technical definition of a pseudo-group of transformations.

\begin{definition}\label{def:amenable2}
A locally finite metric space is said to be amenable if for all $k\geq0$, there is a finite and non-empty subset $A$ of $X$ such that $|cN_k(A)|<2|A|$.
\end{definition}

The second notion of amenability, proposed by Block and Weinberger (see \cite{Bl-We92}, Section 3), is a bit more technical.

\begin{definition}
A metric space $(X,d)$ is said to have coarse bounded geometry if it admits a quasi-lattice; namely, there is $\Gamma\subseteq X$ such that
\begin{enumerate}
\item There exists $\alpha>0$ such that $cN_\alpha(\Gamma)=X$,
\item For all $r>0$ there exists $K_r>0$ such that, for all $x\in X$,
\begin{align}
\left|\Gamma\cap B_r(x)\right|\leq K_r
\end{align}
where $B_r(x)$ stands for the open ball of radius $r$ about $x$.
\end{enumerate}
\end{definition}

\begin{definition}\label{def:amenable3}
A metric space $(X,d)$ with coarse bounded geometry, given by the quasi-lattice $\Gamma$, is said to be amenable if for all $r,\delta>0$, there exists a finite subset $U$ of $\Gamma$ such that
\begin{equation}\label{eq:folner}
\frac{|\partial_rU|}{|U|}<\delta
\end{equation}
where $\partial_rU=\{x\in\Gamma:d(x,U)<r, d(x,\Gamma\setminus U)<r\}$.
\end{definition}

The natural question is about the possible logic implications among these two notions of amenability and the property SN. The following discussion gives an almost complete summary of the situation.\\

Let us start from comparing Definitions \ref{def:amenable2} and \ref{def:amenable3}. Remark 42 in \cite{Ce-Gr-Ha99} contains almost everything and we want to make formal that argument. First of all, the fact that there are locally finite metric spaces that do not have coarse bounded geometry\footnote{For instance, attach over any natural number $n$ a set $A_n$ with $n$ elements and define the metric to be 1 between points in the same $A_n$ and $n+m$ between a point in $A_n$ and one in $A_m$, with $n\neq m$.} and, conversely, there are spaces with coarse bounded geometry that are not locally finite, forces to consider the problem in the restricted class of locally finite spaces with coarse bounded geometry. We have the following result:

\begin{proposition}\label{prop:ceccherinivsblock}
Let $(X,d)$ be a locally finite metric space with coarse bounded geometry. It is amenable in the sense of Definition \ref{def:amenable2} if and only if it is amenable in the sense of Definition \ref{def:amenable3}.
\end{proposition}

\begin{proof}
Let $\Gamma\subseteq X$ be a quasi-lattice. By definition of a quasi-lattice, $X$ is quasi-isometric to $\Gamma$. Since amenability in the sense of Definition \ref{def:amenable2} is invariant under quasi-isometries (see \cite{Ce-Gr-Ha99}, Proposition 38), we can restrict our attention on $\Gamma$. So we need to prove that the F{\o}lner condition on $\Gamma$ in Equation (\ref{eq:folner}) is equivalent to the condition on $\Gamma$ in Definition \ref{def:amenable2}. This is a pretty standard argument that we can summarize as follows. First of all, we observe that amenability in the sense of Definition \ref{def:amenable2} for $\Gamma$ is equivalent to require that for all $k\in\mathbb N$ and for all $\varepsilon>0$ there is a finite non-empty subset $A$ of $\Gamma$ such that $|cN_k(A)|<(1+\varepsilon)|A|$. One implication is indeed trivial; conversely, suppose that for all $k$ there is $A$ such that $|cN_k(A)|<2|A|$ and assume by contradiction that there are $\bar k$ and $\varepsilon$ such that for all $A$ one has $|cN_{\bar k}(A)|\geq(1+\varepsilon)|A|$. Define $\bar K=\bar k n$, where $n$ is a positive integer such that $(1+\varepsilon)^n\geq2$. One gets $|cN_{\bar K}(A)|\geq2|A|$, for all $A$, contradicting the hypothesis. Therefore, it remains to prove that the statement
\begin{align}\label{stat:amenability}
\forall k,\forall\varepsilon,\exists A\subseteq\Gamma : |cN_k(A)|<(1+\varepsilon)|A|
\end{align}
is equivalent to Block-Weinberger's definition of amenability for $\Gamma$. Also in this case one implication is trivial and let us prove only that the condition in Statement (\ref{stat:amenability}) implies Block-Weinberger's amenability. Let $k\in\mathbb N$, $\delta>0$. By Condition (\ref{stat:amenability}), there is a finite subset $A$ of $\Gamma$ such that $|cN_{2k}(A)|<(1+\delta)|A|$. Set $U=cN_k(A)$ and let us prove that it gives Block-Weinberger's amenability. Indeed, first observe that $\partial_kU=cB_{2k}(A)$. Therefore,
$$
\frac{|\partial_kU|}{|U|}=\frac{|cB_{2k}(A)|}{|A|}\cdot\frac{|A|}{|cN_k(A)|}<\delta\cdot1=\delta
$$
as required.

\end{proof}

About the relation between Property SN and amenability in the sense of Definition \ref{def:amenable2} we can give the following

\begin{proposition}\label{prop:caprarovsceccherini}
If $X=(V,E)$ is a locally finite connected graph, then it has the property SN if and only if it is amenable in the sense of Definition \ref{def:amenable2}.
\end{proposition}

\begin{proof}
By Proposition 4.3 in \cite{El-So05} a locally finite connected graph is amenable if and only if

$$
\sup_{k\geq1}\inf\left\{\frac{|cB_k(A)|}{|A|}: A\subseteq X, 0<|A|<\infty\right\}=0
$$

By the observation right after our Definition \ref{def:isoperimetric}, for locally finite connected graphs one has $cB_k(A)=dB_k(A)$ and therefore $X$ is amenable if and only if $\iota(X)=0$.
\end{proof}

For general locally finite metric spaces, property SN and amenability behave quite differently. The following example shows the existence of non-amenable locally finite metric spaces having the property SN.

\begin{example}\label{ex:SNnoamenable}
{\rm Let
\begin{align*}
X=\left\{x_n=\sum_{k=1}^n\frac{1}{k}:n\geq1\right\}
\end{align*}
with the metric induced by $\mathbb R$. This space is certainly locally finite. It is straightforward to show that it has the property SN. Indeed, The sequence of finite subsets $A_n=\{x_1,\ldots,x_n\}$ have the property that $dN_k(A_n)=A_{n+k}$ and therefore $\iota(X)=0$. On the other hand this space is not amenable in the sense of Ceccherini-Silberstein, Grigorchuk and de la Harpe. Indeed, we now show that $|cN_1(A)|\geq2|A|$, for all $A$. Let $|A|=k$ and let $n_A$ be the maximum index $n$ such that $x_n\in A$; thus, $n_A\geq k$. This simple observation implies that
$$
\frac{1}{n_A+1}+\frac{1}{n_A+2}+\ldots+\frac{1}{n_A+k}\leq1
$$
therefore $x_{n_A+1},\ldots,x_{n_A+k}\in cN_1(A)\setminus A$, that shows that $|cN_1(A)|\geq2k$.}
\end{example}

\section{The property SN as a test for embeddings into Hilbert spaces}\label{se:application}

In this section we want to present some applications of the property SN. The spirit of this section is to prove that in many cases the property SN can be used as a test to check the existence of \emph{certain} embeddings into Hilbert spaces; namely, the lack of the property SN is often an obstacle to the existence of \emph{certain} embeddings into Hilbert spaces. In fact, the theory of embedding (isometrically, quasi-isometrically, bilipschitzly,...) a locally finite metric space into a well-understood metric space, as a Banach space or a Hilbert space, has a long and interesting story all over the 20th century and still alive, after the breakthrough papers by Linial, London and Rabinovich \cite{Li-Lo-Ra95} and Guoliang Yu \cite{Yu00} showing unexpected interplay between this theory, Computer Science and K-theory of $C^*$-algebras. In this context it is important to have properties that obstacle the existence of such embeddings, to be used as a test. We believe that the lack of property SN can be an obstacle for the existence of bi-lipschitz embeddings of metric spaces into Hilbert spaces, as well as the lack of amenability is an obstacle to the existence of bi-lipschitz embeddings of a graph into a Hilbert space. In fact, a famous result of Benjamini and Schramm (see \cite{Be-Sc97}, Theorem 1.5) says that a locally finite connected graph of bounded degree with positive isoperimetric constant contains a tree with positive isoperimetric constant such that the inclusion map is a bi-lipschitz embedding. It is now a classical result of Bourgain (see \cite{Bo86}, Theorem 1) that such trees do not have bi-lipschitz embeddings into Hilbert spaces. In this sense, the lack of amenability is an obstacle to the existence of bi-lipschitz embeddings into Hilbert spaces. Observe that if we try to have a similar result for locally finite metric spaces that are not necessarily graphs, we encounter some problems. E.g., the space of Example \ref{ex:SNnoamenable} is not amenable in the sense of Definition \ref{def:amenable2}, but it is isometrically embedded into a Hilbert space.\\

We believe that, in order to find such an obstacle, one should replace amenability with property SN. In this section we prove some results in this direction. The first one concerns isometric embeddings into Hilbert spaces. Later we will be able to give some results about bi-lipschitz embeddings, but we will be forced to restrict the target space considering only finite-dimensional Hilbert spaces.\\

With the following definition, we generalize the notion of metric tree proposed by Ghys and de la Harpe at pp.27-28 in \cite{Gh-dlHa90}.\\

Let $X=(V,E)$ be a connected simple graphs (i.e. without loops and multiple edges) of bounded degree such that every edge $xy$ is labeled with a positive real number denoted by $d(x,y)$. We also require the following two compatibility properties:

\begin{enumerate}
\item If $x_0,x_1,\ldots,x_{n-1},x_n$ and $y_0,y_1,\ldots,y_{n-1},y_n$ are shortest paths in the unlabeled graph connecting the same points (i.e. $x_0=y_0$ and $x_n=y_n$), then

\begin{align}\label{eq:shortpath}
\sum_{i=1}^nd(x_{i-1},x_i)=\sum_{i=1}^nd(y_{i-1},y_i)
\end{align}

\item If $x_0,x_1\ldots,x_{n-1},x_n$ is a shortest path in the unlabeled graph connecting $x$ to $y$ and $y_0,y_1,\ldots,y_{m-1},y_m$, with $m>n$, is another path connecting $x$ to $y$, then
\begin{align}\label{eq:metriccompatibility}
\sum_{i=1}^nd(x_{i-1},x_i)<\sum_{i=1}^md(y_{i-1},y_i)
\end{align}

\end{enumerate}

It is clear that one can define a metric on the set of vertices just adding the various weights encountered along a shortest path. Conditions (\ref{eq:shortpath}) and (\ref{eq:metriccompatibility}) gives some compatibility between the metric and the underlying graph structure: the metric is additive only along shortest paths. If the underlying graph is a tree labeled with positive real numbers, then we obtain metric trees in the sense of Ghys and de la Harpe and our additional properties are automatically satisfied.

\begin{definition}\label{def:tripods}
A tripod is a graph given by a vertex $v\in V$ with three neighbors $v_1,v_2,v_3$ such that there is no connection between the $v_i$'s. A semi-tripod is a graph given by a vertex $v$ with three neighbors $v_1,v_2,v_3$ such that there is at most one connection between the $v_i$'s. A metric tripod (resp. metric semi-tripod) is a tripod (resp. semi-tripod) equipped with positive weights verifying the compatibility properties in (\ref{eq:shortpath}) and (\ref{eq:metriccompatibility}).
\end{definition}

\begin{lemma}\label{lem:tripod}
Metric tripods and metric semi-tripods cannot be isometrically embedded into a Hilbert space.
\end{lemma}

\begin{proof}
We prove that semi-tripods are not isometrically embeddable into a Hilbert space. The proof for tripods is even simpler. Let us suppose that $v_1$ and $v_2$ are connected and let $\beta$ be their distance. Let us denote by $\alpha_i=d(v,v_i)$. By the definition of the metric, we have $d(v_2,v_3)=\alpha_2+\alpha_3$ and $d(v_1,v_3)=\alpha_1+\alpha_3$. Suppose, by contradiction, that the semi-tripod is isometrically embedded into a Hilbert space. Since in a Hilbert space, the metric is additive only along lines, it follows that $v_1,v,v_2,v_3$ belong to the same line and therefore $\beta=|\alpha_1-\alpha_2|$. Suppose, for instance, $\alpha_2\geq\alpha_1$, then one would have that the path $v_3,v,v_1,v_2$ has length $\alpha_3+\alpha_1+\alpha_2-\alpha_1=\alpha_3+\alpha_2$ which is the same as the length of the path $v_3,v,v_2$. This contradicts the compatibility axiom in (\ref{eq:metriccompatibility}).
\end{proof}

\begin{theorem}\label{prop:emedding}
Let $X$ be a metric graph verifying (\ref{eq:shortpath}), (\ref{eq:metriccompatibility}) and such that the resulting metric is locally finite. If $X$ embeds isometrically into a Hilbert space, then it has the property SN.
\end{theorem}

\begin{proof}
By Lemma \ref{lem:tripod}, it suffices to show that a graph without SN contains tripods or semi-tripods. Denote by $X^u=(V,E^u)$ the underlying connected graph; i.e. the connected graph obtained by taking out the labels from $X$. Let $d_u$ be the shortest-path distance on $X^u$. Given $A\subseteq V$ and $k\geq1$, let $cN_k(A)=\{v\in V: d_u(v,A)\leq k\}$ and $cB_k(A)=cN_k(A)\setminus A$. We first show that for all finite subsets $A$ of $X$ and for all $k\geq1$, one has $dB_k(A)\subseteq cB_k(A)$ (notice that $dB_k$ refers to $d$ and $cB_k$ refers to $d_u$). Indeed, let $v\in dB_k(A)$ and let $v_0,v_1,\ldots,v_l$ be such that $v_0\in A$, $v_l=v$, $l\leq k$ and

\begin{align}\label{eq:firstchain}
0=d(v_0,A)<d(v_1,A)<\ldots<d(v_l,A)
\end{align}

is a complete chain in $P=\{d(v,A): v\in V\}$. Now, let $\gamma$ be a shortest path in $X^u$ connecting $v_l$ to $A$ and let $k'$ be its length. We have to prove that $k'\leq k$. Let $w_0,w_1,\ldots,w_{k'}$, with $w_{k'}=v_l=v$, be the vertices encountered along this shortest path. We will now show that
\begin{align}\label{eq:secondchain}
0=d(w_0,A)<d(w_1,A)<\ldots<d(w_{k'},A)
\end{align}
By contradiction, suppose that there is $i$ such that $d(w_i,A)\leq d(w_{i-1},A)$. By definition of the metric $d$, which is just the sum of the various edges encountered along a shortest path, and by the compatibility axioms (\ref{eq:shortpath}) and (\ref{eq:metriccompatibility}) a shortest path connecting $w_i$ to $A$ must have length strictly less than $i$. This contradicts the fact that $\gamma$ is a shortest path.

Now, there are $k'+1$ numbers appearing in (\ref{eq:secondchain}) and, moreover, the first one and the last one equal respectively the first one and the last one in (\ref{eq:firstchain}). The numbers appearing in (\ref{eq:firstchain}) form a complete chain and this implies that the numbers appearing in (\ref{eq:secondchain}) must be less; i.e. $k'+1\leq k+1$, as required.\\

Now, we have assumed that $X$ does not have the property SN and therefore, there is $k\geq1$, such that

$$
\inf\left\{\frac{|dB_k(A)|}{|A|}:A\subseteq V,0<|A|<\infty\right\}>0
$$

Now, by the previous step, we know that $dB_k(A)\subseteq cB_k(A)$ and therefore the unlabeled graph $X^u$ has positive $k$-isoperimetric constant

$$
\iota_k(X^u)=\inf\left\{\frac{|cB_k(A)|}{|A|}:A\subseteq V,0<|A|<\infty\right\}>0
$$

Now, $X^u$ has bounded degree and therefore, by a standard argument, also its isoperimetric constant is positive. It now suffices to show that any connected graph with positive isoperimetric constant contains a tripod or a semi-tripod. Let $v_0$ be a fixed vertex in $X^u$ and consider the sets $A_n=dB_n(v_0)$. Since the isoperimetric constant is positive, there are $n\geq2$ and $v\in dB_n(v_0)$ such that $v$ is connected to at least two vertices $v_1,v_2$ in $dB_{n+1}(v_0)$. By construction $v$ is also connected to some $v_3\in dB_{n-1}(v_0)$. Now, $v_1$ and $v_2$ cannot be connected to $v_3$, otherwise they would belong to $dB_{n}(v_0)$. Therefore, we have found a tripod or semi-tripod $C=\{v,v_1,v_2,v_3\}$.
\end{proof}

Observe that this proposition would have been false replacing the property SN with amenability in the sense of Definition \ref{def:amenable2}, as Example \ref{ex:SNnoamenable} shows.\\\\

Now we want to prove two results concerning bi-lipschitz embeddings into finite-dimensional Hilbert spaces.

\begin{definition}\label{def:uniformboundedgeometry}
A locally finite metric space $(X,d)$ is said to have uniform bounded geometry if there is a constant $C>0$ such that for all $r>0$ and for all $x,y\in X$ one has
$$
\frac{|B(x,r)|}{|B(y,r)|}\leq C
$$
where $B(x,r)$ stands for the closed ball of radius $r$ about $x$.
\end{definition}

We recall that a locally finite metric space $X$ is said to have bounded geometry if for all $r>0$, there is a constant $C_r$ such that for all $x\in X$, one has $|B(x,r)|\leq C_r$. It is then clear that a space with uniform bounded geometry has also bounded geometry. The converse is not true. We will use uniform bounded geometry mainly through the following simple

\begin{lemma}\label{lem:uniformbounded}
Let $X$ be an infinite locally finite metric space with uniform bounded geometry. Suppose that there are a sequence of points $x_n\in X$, a sequence of radii $r_n>0$ and a constant $M>0$ such that $|B(x_n,r_n)|\leq M$, for all $n$. Then $\limsup_{n\rightarrow\infty}r_n<\infty$.
\end{lemma}

\begin{proof}
Suppose, by contradiction, that $\limsup_{n\rightarrow\infty}r_n=\infty$. Fix arbitrarily $x\in X$ and $r>0$. Let $n$ such that $r_n\geq r$. Since $X$ has uniform bounded geometry, one has
$$
\frac{|B(x,r)|}{|B(x_n,r_n)|}\leq\frac{|B(x,r_n)|}{|B(x_n,r_n)|}\leq C
$$
therefore $|B(x,r)|\leq CM$, where the constants $C$ and $M$ are universal. This means that the space is bounded and therefore it is finite.
\end{proof}

Let $x\in X$, denote by

\begin{align}
\Gamma^x=\{r>0:|B(x,r)|>|B(x,s)|, \text{ for all } s<r\}
\end{align}

Roughly speaking, $\Gamma^x$ is the set of radii such that the corresponding ball about $x$ gets bigger. Since $(X,d)$ is locally finite, the sets $\Gamma^x$ are countable. Let us suppose that their elements are listed in increasing order; i.e. we will write $\Gamma^x=\{r_1^x,r_2^x,\ldots\}$, where $r_n^x<r_{n+1}^x$, for all $n$.

\begin{definition}\label{def:polynomialgrowthspace}
The space $(X,d)$ is said to have polynomial increment at point $x$ if there is $C>0$ such that
$$
r_n^x\leq n^C
$$
for all $n$.
\end{definition}

\begin{theorem}\label{th:polygrowth}
Let $(X,d)$ be a locally finite space with uniform bounded geometry and polynomial increment in at least one point $x$. If $(X,d)$ has a bi-lipschitz embedding into a finite-dimensional Hilbert space, then it has the property SN.
\end{theorem}
\begin{proof}
Suppose, by contradiction, that $(X,d)$ has a bi-lipschitz embedding into a finite-dimensional Hilbert space and does not have the property SN. Denote by $B_n$ the closed ball of radius $r_n^x$ about $x$. Since $(X,d)$ does not have the property SN, there is $k\geq1$ such that

\begin{align}
\inf\left\{\frac{|dN_k(B_n)|}{|B_n|}:n\in\mathbb N\right\}>1
\end{align}

Let $\varepsilon>0$ such that $|dN_k(B_n)|\geq(1+\varepsilon)|B_n|$, for all $n$. Let us observe explicitly two simple facts:\\

\begin{enumerate}
\item $dN_k(B_n)=B_{n+k}$. This follows basically by definition of complete chain and of the numbers $r_n^x$.
\item For $r\in\mathbb N$, denote by $\Gamma^x_r=\Gamma^x\cap[2^r,2^{r+1}]$. Then one has $\lim\sup_{r\rightarrow\infty}|\Gamma^x_r|=\infty$. This follows from the fact that the sequence $r_n^x$ has polynomial increment, while the intervals $[2^r,2^{r+1}]$ have exponential increment.\\
\end{enumerate}

The first observation allows the following chain of inequality:

\begin{align}\label{eq:chain}
|B(x,2^{r+1})|=|dN_k(B(x,2^{r+1}-\alpha_1))|\geq(1+\varepsilon)|B(x,2^{r+1}-\alpha_1)|\geq\ldots\geq(1+\varepsilon)^s|B(x,2^r)| \end{align}

where the $\alpha_i$'s are chosen in the obvious way. Observe that $s$ depends on $r$ but, from the second of the previous observations, can be made arbitrarily big with $r$. Now $X$ has uniform bounded geometry and then it has also bounded geometry and therefore, the following number is well-defined: $K_r=\max\{|B(x,2^r)|:x\in X\}$. Since the space has uniform bounded geometry, we have $|B(x,2^r)|\geq C^{-1}K_r$, for all $x$, and therefore, using (\ref{eq:chain}) we have

\begin{align}\label{eq:nodoubling}
|B(x,2^{r+1})|\geq(1+\varepsilon)^sC^{-1}K_r
\end{align}

where it is important to keep in mind that $s$ can be made arbitrarily big with $r$. We want to show that this is a contradiction. Indeed, it is well-known that if a metric space $(X,d)$ has a bi-lipschitz embedding into a finite-dimensional Hilbert space, then $X$ has the property that there is a universal constant $L$ such that every ball of radius $2t$ can be covered by $L$ balls of radius $t$ (for a proof see, for instance, Lemma 2.1. in \cite{La-Pl01}). Applying this property with $t=2^r$ and majorizing the cardinality of each ball of radius $2^r$ with $K_r$, one has $|B(x,2^{r+1})|\leq LK_r$, where $L$ does not depend on $r$. This contradicts (\ref{eq:nodoubling}).
\end{proof}

Now we use the proof of Theorem \ref{th:polygrowth} to prove a Bourgain-like theorem for metric trees. Indeed, Bourgain's theorem says that a tree with positive isoperimetric constant does not have any bi-lipschitz embedding into a Hilbert space. It is then natural to ask the following question: is it true that a metric tree without property SN does not have any bi-lipschitz embedding into Hilbert spaces? In this generality, the answer is clearly negative. Indeed, as observed by Sergei Ivanov in \cite{MO2}, the 2-regular rooted tree with weights $10^k$, where $k$ denotes the number of edge that one needs to reach the root, has a bi-lipschitz embedding into the real line (and it does not have the property SN). The question becomes more interesting if we add the hypothesis that the metric tree has uniform bounded geometry.

\begin{problem}
Let $T$ be an infinite metric tree with uniform bounded geometry and without property SN. Is it true that it does not have any bi-lipschitz embedding into a Hilbert space?
\end{problem}

A positive answer would be a cute generalization of Bourgain's theorem. We now give a positive answer in the case when only embeddings into finite-dimensional Hilbert spaces are allowed.

\begin{theorem}\label{prop:metrictree}
Let $T$ be an infinite metric tree with uniform bounded geometry and without property SN. Then $T$ does not admit any bi-lipschitz embedding into a finite dimensional Euclidean space.
\end{theorem}

\begin{proof}
By the proof of Theorem \ref{th:polygrowth}, it suffices to show that there is a vertex $x$ such that $\lim\sup_{r\rightarrow\infty}|\Gamma^x_r|=\infty$. We will actually prove that this is true for all $x$. By contradiction, let $|\Gamma^x_r|\leq M_x$, for all $r$. Let $N>0$ and choose $k(N)$ big enough such that the interval of natural numbers $[r_{k(N)}^x-N,r_{k(N)}^x+N]$ intersects at most two intervals of the shape $[2^r,2^{r+1}]$. Let $y_N\in T$ be a vertex at distance $r_{k(N)}^x$ from $x$. We have
\begin{align}\label{eq:inequality}
\left|\Gamma^x\cap\left[r_{k(N)}^x-N,r_{k(N)}^x+N\right]\right|\leq2M_x
\end{align}
Also, we may suppose that $k(N)$ is chosen in such a way that $r_{k(N)}^x=d(y_N,x)>N$. This will be useful later.\\

Now, to simplify the notation, let us denote $\Gamma^{y_N}=\{r_1^N,r_2^N,\ldots\}$. By Lemma \ref{lem:uniformbounded}, there is $R_1$ such that $r_1^N\leq R_1$, for all $N$. Similarly, there is $R_2$ such that $r_2^N\leq R_2$, for all $N$. Indeed, assume by contradiction that $\limsup_{N\rightarrow\infty}r_2^N=\infty$. We have $B(y_N,r_2^N-1)=B(y_N,R_1)$. By the bounded geometry of $T$ we know that there is a constant $C'$ such that $|B(y_N,R_1)|\leq C'$ for all $N$ and so we would have a sequence $B(y_N,r_2^N-1)$ of arbitrarily large balls having cardinality at most $C'$. This is not possible thanks to an obvious application of Lemma \ref{lem:uniformbounded}. It is now clear that by induction we get that for all $h\in\mathbb N$, there is $R_h$ such that $r_h^N\leq R_h$, for all $N$.\\

After these preliminary observations, we now study the structure of the ball $B(y_N,N)$. Let $z\in B(y_N,N)$. Since $T$ is a tree and $d(y_N,x)>N$ we have only two possibilities:
\begin{enumerate}
\item The shortest path connecting $x$ to $y_N$ can be extended to a shortest path connecting $x$ to $z$;
\item The shortest paths connecting $x$ to $z$ and $x$ to $y_N$ intersect up to a certain vertex $\overline z$.
\end{enumerate}
Let us consider the first case. Let $\ell$ be the length of the shortest path connecting $y_N$ to $z$. By definition of the sequence $r_n^x$, we have $d(x,z)\geq r_{k(N)+\ell}^x$. Therefore, since the distance is additive along shortest paths, we have
\begin{align}\label{eq:firstcase}
d(y_N,z)=d(x,z)-d(x,y_N)\geq r_{k(N)+\ell}^x-r_{k(N)}^x
\end{align}
Now, let us consider the second case. Let $\overline\ell$ be the length of the shortest path connecting $y_n$ to $\overline z$. By definition of the sequence $r_n^x$, we have $d(x,\overline z)\leq r_{k(N)-\overline\ell}^x$. Therefore, since the metric is additive along shortest paths, we have
\begin{align}\label{eq:secondcase}
d(z,y_N)=d(z,\overline z)+d(x,y_N)-d(x,\overline z)\geq d(x,y_N)-d(x,\overline z)\geq r_{k(N)}^x-r_{k(N)-\overline\ell}^x
\end{align}
Now, we know that $d(y_N,z)\leq N$ and therefore, by the inequalities (\ref{eq:firstcase}) and (\ref{eq:secondcase}), one of the following holds:
\begin{align*}
\text{Either }\qquad N\geq r_{k(N)+\ell}^x-r_{k(N)}^x\qquad\text{ or }\qquad N\geq r_{k(N)}^x-r_{k(N)-\overline\ell}^x
\end{align*}
In other words, one of the following holds:
\begin{align*}
\text{Either }\qquad r_{k(N)+\ell}^x\leq r_{k(N)}^x+N\qquad\text{ or }\qquad r_{k(N)-\overline\ell}^x\geq r_{k(N)}^x-N
\end{align*}
Together with (\ref{eq:inequality}), this means that there are only $2M_x$ free slots for $\ell$ and $\overline\ell$; namely,
$$
B(y_N,N)\subseteq B(y_N,r_{2M_x}^N)
$$
Now, we know, by the first step of the proof, that there is a uniform bound $r_{2M_x}^N\leq R_{2M_x}$, for all $N$. Thus, for all $N$, we have
$$
B(y_N,N)\subseteq B(y_N,R_{2M_x})
$$
Now, since $T$ has bounded geometry, there is a constant $C'$ depending only on $2M_x$ such that $B(y_N,R_{2M_x})\leq C'$, for all $N$. Therefore, $|B(y_N,N)|\leq C'$, for all $N$. This is a contradiction, thanks to Lemma \ref{lem:uniformbounded}.
\end{proof}

\begin{note}\label{note:referencemetrictree}
{\rm Ours is not the first result concerning bilipschitz embeddings of metric trees into Hilbert spaces. Lee, Naor and Peres proved that a metric tree has a bilipschitz embedding into a Hilbert space if and only if it does not contain arbitrarily large complete binary trees with uniform bounded distortion (see \cite{Le-Na-Pe09}, Theorem 1.1). Gupta, Krauthgamer and Lee proved that a metric tree admits a bi-lipschitz embedding into a finite dimensional Hilbert space if and only if there is a constant $C$ such that every ball can be covered by $C$ balls of half radius (see \cite{Gu-Kr-Le02}, Theorem 2.5, and \cite{Le-Na-Pe09}, Theorem 2.12).}
\end{note}

\begin{remark}\label{rem:coarseSN}
{\rm A natural question is whether the box-space defined by a countable family of connected regular graphs has property SN if and only if this family is an expander. The answer is strongly negative. Indeed, as observed by Antoine Gournay, every family of expanders does have property SN. Let indeed $G_n$ be a family of expanders and let $G=\bigsqcup G_n$ be a box-space defined by the $G_n$'s. We recall that the box-space is defined just putting a metric on the disjoint union of the $G_n$'s which extends the graph-metric on each $G_n$ and verifies the property that $d(G_n,G_m)\to\infty$, as $m+n\to\infty$. Fix $k\geq1$. Since it is known that the diameter of the $G_n$'s increases logarithmically (see, for instance, p.455 in \cite{Ho-Li-Wi06}), we may find $n_k$ such that $\text{diam}(G_n)\geq k$, for all $n\geq n_k$, and $d(G_n,G_{n+1})>k$, for all $n\geq n_k$. Define a sequence of sets $(F_k^n)_{n\geq n_k}$ by
\begin{align*}
F_k^n=\bigcup_{i=1}^n G_i\cup\left(G_{n+1}\setminus B(x,k)\right)
\end{align*}
where $x$ is arbitrarily chosen inside $G_{n+1}$. By the construction, it follows that $cN_k(F_k^n)=\bigcup_{i=1}^{n+1}G_i$, for all $n\geq n_k$. Therefore
$$
\iota_k(G)=\inf\left\{\frac{|dB_k(A)|}{|A|} : A\subseteq G, 0<|A|<\infty \right\}=0
$$
for all $k$, since it goes to zero along the sequence $F_k^n$.\\

This bad behavior put in evidence the fact that Property SN is not a good coarse property and, in order to study coarse embeddings, one should use some other property. A very recent e-print by Brodzki, Niblo, \v{S}pakula, Willett and Wright contains a new property, called uniform local amenability, that might do the job (see Definition 2.2 in \cite{BNSWW12}). Moreover, their property can be easily rewritten in terms of a \emph{coarse isoperimetric constant} opening the way to a series of questions that merit attention in further researches.}
\end{remark}

\section{The zoom isoperimetric constants}\label{se:zoomiso}

In this section we want to introduce another invariant, which is, in some sense, more precise than the isoperimetric constant. In fact, the isoperimetric constant looks at the best behavior inside the metric space and so the information encoded in it could be very far from what a local observer really sees. In order to make this obscure sentence clearer, let us consider the following example. Consider a regular tree of degree $d\geq3$ and attach over a vertex a copy of the integers. We get a connected graph of bounded degree with isoperimetric constant equal to zero. This graph is then amenable in every sense: it has the property SN and it is amenable in the sense of Definitions \ref{def:amenable2} and \ref{def:amenable3}. This happens because the isoperimetric constant forgets the non-amenable part. On the other hand, consider an observer living in the vertex $x$. His/her way to observe the universe where (s)he lives is to \emph{increase his/her knowledge step by step}, considering at the first step the set $A_0=\{x\}$, then $A_1=dN_1(\{x\})$ and so on: his/her point of view on the universe is given by the sequence $A_n=dN_n(\{x\})$. It is clear that, wherever $x$ is, the sequence $A_n$ is eventually doubling, in the sense that $|A_n|\geq2|A_{n-1}|$. So, the isoperimetric constant says that the universe is amenable, but every local observer would say that the universe is non-amenable. This lack of agreement is the reason why in this section we try to introduce a new invariant which takes into account all these local observations.\\

Let $(X,d)$ be a locally finite metric space, $x\in X$ and $k,n$ be two positive integers. As for the isoperimetric constant, we fix $k$ and we are going to look at the sequence, in the parameter $n\geq1$, of the ratios $\frac{|dN_{nk}(x)|}{|dN_{(n-1)k}(x)|}$. There are now two different ways to take out information from this sequence, one is taking the infimum, the other one is taking the limsup. Therefore, denote by

\begin{align}
\underline\zeta_k(x)=\inf\left\{\frac{|dN_{nk}(x)|}{|dN_{(n-1)k}(x)|}, n\geq1\right\}
\end{align}
\begin{align}
\overline\zeta_k(x)=\limsup_{n\rightarrow\infty}\frac{|dN_{nk}(x)|}{|dN_{(n-1)k}(x)|}
\end{align}



Now, denote by
$$
\underline\zeta(x)=\sup_{k\geq1}\underline\zeta_k(x)\qquad\overline\zeta(x)=\sup_{k\geq1}\overline\zeta_k(x)
$$
and then
$$
\underline\zeta^+(X)=\sup_{x\in X}\underline\zeta(x)\qquad\overline\zeta^+(X)=\sup_{x\in X}\overline\zeta(x),x\in X$$

Observe that

\begin{align}
\iota(X)\leq\underline\zeta^+(X)-1\leq\overline\zeta^+(X)-1
\end{align}

In particular, if $X$ is the Cayley graph of a finitely generated group, then $\underline\zeta^+(X)=1$ or $\overline\zeta^+(X)=1$ implies amenability. The converse is very far from being true. The properties $\underline\zeta^+(X)=1$ and $\overline\zeta^+(X)=1$, are stronger than amenability and they are related to the growth of a group, as we now show. Let $G$ be a finitely generated group, fix a finite symmetric generating set $S$ and consider the locally finite metric space $(G,d_S)$, where $d_S$ denotes the word-metric induced by $S$. Denote by $B(n)$ the closed ball of radius $n$ about the identity. It is well-known that the sequence $\sqrt[n]{|B(n)|}$ is always convergent (see, for instance, \cite{Ha00}, VI.C, Proposition 56). We recall the following definition:

\begin{definition}
A group is said to have sub-exponential growth if $\lim_{n\rightarrow\infty}\sqrt[n]{B(n)}=1$. Otherwise, it is said to have exponential growth.
\end{definition}

It is well-known that there is no perfect correspondence between amenability and growth rate of a group. Indeed, groups of sub-exponential growth are amenable, but there are amenable groups generated by $m$ elements whose growth approach the one of the free group on $m$ generators (see \cite{Ar-Gu-Gu05}). There is instead a \emph{perfect} correspondence using the property that $\underline\zeta^+(G)=1$.

\begin{proposition}\label{prop:locallyamenable}
A group $G$ has sub-exponential growth if and only if $\underline\zeta^+(G)=1$.
\end{proposition}

\begin{proof}
Let $G$ be a group with exponential growth. It follows that $\lim_n\sqrt[n]{B(n)}=a>1$ and then $|B(n)|\sim a^n$ which implies that
\begin{align}
\inf\left\{\frac{|B(n+1)|}{|B(n)|},n\in\mathbb N\right\}>1
\end{align}
It follows that $\underline\zeta(e)>1$, so $\underline\zeta^+(G)>1$.
\\\\
Let us prove the converse. This can be proved making use of deep results, as in the proof of the next Proposition. We thank Martin Kassabov for providing the following elementary proof in \cite{MO1}. Suppose by contradiction, that there exists $x\in G$ such that $\underline\zeta_k(x)>1$. By homogeneity of the Cayley graph of a finitely generated group, we may suppose that $x=e$, the identity in $G$; i.e., $\underline\zeta_k(e)=c>1$. By induction, one easily gets that $B(nk)\geq c^n$. Then

$$
\lim\sup_s\sqrt[s]{B(s)}\geq\lim\sup_n\sqrt[nk]{B(nk)}\geq\lim\sup_n\sqrt[nk]{c^n}=\sqrt[k]{c}>1
$$

showing that $G$ has exponential growth.
\end{proof}

The study of the property $\overline\zeta^+(X)=1$, when $X$ is the Cayley graph of a finitely generated group, leads to a problem that seems to be open. With the notation above, let us recall the following

\begin{definition}\label{def:polynomialgrowth}
A finitely generated group is said to have polynomial growth if there are a positive integer $k$ and a positive real number $C$ such that $B(n)\leq Cn^k$, for all $n$.
\end{definition}

\begin{proposition}\label{prop:polygrowth}
Let $X$ be the Cayley graph of a finitely generated group $G$.
\begin{enumerate}
\item If $G$ has polynomial growth, then $\overline\zeta^+(X)=1$.
\item If $\overline\zeta_k(X)=1$ for all $k$, then $G$ has sub-exponential growth.
\end{enumerate}
\end{proposition}

\begin{proof}
\begin{enumerate}
\item By a celebrated result, due to many authors, a group of polynomial growth has growth which is exactly polynomial; i.e. $|B(n)|=Cn^d+o(n^d)$ (see \cite{Gu71}, \cite{Ba72}, \cite{Gr81}, \cite{Pa83}). It follows, by direct computation, that $\overline\zeta^+(X)=1$.
\item This proof is basically the same as the first part of the previous proposition.
\end{enumerate}
\end{proof}

It would be interesting to understand what can happen for groups of sub-exponential growth that do not have polynomial growth.

\begin{problem}\label{prob:intermediategrowth}
What can we say about $\overline\zeta_k(X)$, when $X$ is the Cayley graph of a finitely generated group of intermediate growth?
\end{problem}

This question in fact arose in a series of comments by Ben Green, Martin Kassabov and the author to author's MathOverflow question \cite{MO1}. As argued by Martin Kassabov and checked by the author himself, the estimations recently proved in \cite{Ba-Er10} and \cite{Ka-Pa11} are not strong enough to find either an example of a group of intermediate growth such that $\overline\zeta_k(X)=1$, for all $k$, or an example of a group of intermediate growth such that $\overline\zeta_k(X)>1$, for at least an index $k$.


\begin{thebibliography}{9}
\bibitem[Ar-Gu-Gu05]{Ar-Gu-Gu05} Arzhantseva, G.N., Guba, V.S., Guyot, L. \emph{Growth rates of amenable groups}, J. Group Theory, 8 (2005), no.3, 389-394.
\bibitem[Ba-Er10]{Ba-Er10} Bartholdi L., Erschler A. \emph{Growth of permutational extensions}, to appear in Inventiones Mathematicae. Available online at http://arxiv.org/abs/1011.5266.
\bibitem[Ba72]{Ba72} Bass, H. \emph{The degree of polynomial growth of finitely generated nilpotent groups}, Proc. London Math. Soc. (3) 25 (1972), 603–614.
\bibitem[Be-Sc97]{Be-Sc97} Benjamini I. and Schramm O. \emph{Every graph with a positive Cheeger constant contains a tree with a positive Cheeger constant}, GAFA, Geom. Funct. Anal. 7 (1997), 403-419.
\bibitem[Bi-Mo-ST88]{Bi-Mo-ST88} Biggs, N.L., Mohar, B. and Shawe-Taylor, J. \emph{The spectral radius of infinite graphs}, Bull. London Math. Soc. 20 (1988), 116-120.
\bibitem[Bl-We92]{Bl-We92} Block, J. and Weinberger, S. \emph{Aperioding tilings, positive scalar curvature and amenability of spaces}, J. of the Amer. Math. Soc. 5 (1992), 907-918.
\bibitem[Bo86]{Bo86} Bourgain J. \emph{The metrical interpretation of superreflexivity in Banach spaces}, Israel J. Math. vol. 56, No. 2 (1986) 222-230.
\bibitem[BNSWW12]{BNSWW12} Brodzki J., Niblo G., Spakula J., Willett R. and
Wright N. \emph{Uniform Local Amenability}, Preprint 2012.
\bibitem[Ce-Co10]{Ce-Co10} Ceccherini-Silberstein, T., Coornaert, M. \emph{Cellular automata and groups}, Springer Monographs in Mathematics, Springer-Verlag, Berlin (2010)
\bibitem[Ce-Gr-Ha99]{Ce-Gr-Ha99} Ceccherini-Silberstein, T., Grigorchuk, R.I., de la Harpe, P. \emph{Amenability and paradoxical decomposition for pseudogroups and for discrete metric spaces}, Proc. Steklov Inst. Math. 224 (1999), 57-97.
\bibitem[De-Si-So95]{De-Si-So95} Deuber, W.A., Simonovitz, M. and S\'{o}s, V.T. \emph{A note on paradoxical metric spaces}, Studia Sci. Math. Hungar. 30 (1995), 17-23.
\bibitem[El-So05]{El-So05} Elek, G and S\'{o}s, V.T. \emph{Paradoxical decompositions and growth conditions}, Combin. Probab. Comput. 14 (2005), no.1-2, 81-105.
\bibitem[Gh-dlHa90]{Gh-dlHa90} Ghys, \'{E}., de la Harpe, P. \emph{Sur les groupes hyperboliques d'après Mikhael Gromov}, Progress in Mathematics, 83. Birkhäuser Boston, Inc., Boston, MA, 1990. xii+285 pp. ISBN 0-8176-3508-4
\bibitem[Gr81]{Gr81} Gromov, M. \emph{Groups of polynomial growth and expanding maps}, Inst. Hautes Études Sci. Publ. Math. No. 53 (1981), 53-73.
\bibitem[Gr-La-Pa81]{Gr-La-Pa81} Gromov M., Lafontaine J., Pansu P. \emph{Structure m\'{e}trique pour le vari\'{e}t\'{e}s riemanniennes}, Cedic/F.Nathan (1981).
\bibitem[Gu71]{Gu71} Guivarc'h, Y. \emph{Groupes de Lie à croissance polynomiale}, C. R. Acad. Sci. Paris Sér. A-B 272 (1971), A1695–A1696. 22E15.
\bibitem[Gu-Kr-Lee02]{Gu-Kr-Le02} Gupta A., Krauthgamer R., Lee J.R. \emph{Bounded geometries, fractals, and low-distortion embeddings}, In 44th Symposium on Foundations of Computer Science, pp. 534-543, 2002.
\bibitem[Ha00]{Ha00} de la Harpe, P. \emph{Topics in Geometric Group Theory}, Chicago Lectures in Mathematics, University of Chicago Press, Chicago, IL, 2000.
\bibitem[Ho-Li-Wi06]{Ho-Li-Wi06} Hoory S., Linial N., Widgerson, A. \emph{Expander graphs and their applications}, Bull. Amer. Math. Soc. vol. 43, No. 4 (2006) 439-561.
\bibitem[Ka-Pa11]{Ka-Pa11} Kassabov M., Pak I. \emph{Groups of oscillating intermediate growth}, preprint (2011), available online at http://arxiv.org/abs/1108.0262.
\bibitem[La-Pl01]{La-Pl01} Lang U., Plaut C. \emph{Bilipschitz embeddings of metric spaces into Space Forms}, Geometriae Dedicata, vol. 87, No.1-3 (2001), 285-307.
\bibitem[Le-Na-Pe09]{Le-Na-Pe09} Lee J.R., Naor A., Peres Y. \emph{Trees and Markov convexity}, Geom. Funct. Anal. 18 (2009) 1609-1659.
\bibitem[Li-Lo-Ra95]{Li-Lo-Ra95} Linial, N. London, E. and Rabinovich, Yu. \emph{The geometry of graphs and some of its algorithmic applications}, Combinatorica (1995) 15, 215-245.
\bibitem[McMu89]{McMu89} McMullen, C. \emph{Amenability, Poincar\'{e} series and quasiconformal maps}, Invent. Math. 97 (1989), 95-127.
\bibitem[MO1]{MO1} MathOverflow, http://mathoverflow.net/questions/80537.
\bibitem[MO2]{MO2} MathOverflow, http://mathoverflow.net/questions/88733.
\bibitem[Pa83]{Pa83} Pansu, P.
\emph{Croissance des boules et des g\'{e}od\'{e}siques ferm\'{e}es dans les nilvari\'{e}t\'{e}s},
Ergodic Theory Dyn. Syst. 3, 415-445 (1983).
\bibitem[Yu00]{Yu00} Yu, G. \emph{The coarse Baum-Connes conjecture for spaces which admit a uniform embedding into Hilbert space}, Invent. Math., Vol. 139, 1 (2000) 201-240
\end{thebibliography}
\end{document}